\documentclass[11pt]{article}
\usepackage{fullpage}
\usepackage{amsmath,graphicx,amsthm,amsfonts,amssymb,amstext}
\usepackage{ragged2e, enumerate,graphicx,lscape}
\usepackage{pgf, tikz, float,subcaption}
\usetikzlibrary{graphs, graphs.standard}
\usepackage{subfiles}
\usepackage{hyperref}

\newtheorem{theorem}{Theorem}

\newtheorem{conjecture}{Conjecture}

\theoremstyle{remark}

\usepackage{hyperref}
\hypersetup{
	colorlinks=true,
	linkcolor=blue,
    citecolor=blue,
    }

\DeclareMathOperator{\tr}{tr}

\title{A Linear Lower Bound for the Square Energy of Graphs}
\author{Saieed Akbari \and Hitesh Kumar \and Bojan Mohar \and Shivaramakrishna Pragada}

\date{}

\begin{document}
\maketitle

\begin{abstract}
Let $G$ be a graph of order $n$ with eigenvalues $\lambda_1 \geq \cdots \geq\lambda_n$. Let 
\[s^+(G)=\sum_{\lambda_i>0} \lambda_i^2, \qquad  s^-(G)=\sum_{\lambda_i<0} \lambda_i^2.\] 
The smaller value, $s(G)=\min\{s^+(G), s^-(G)\}$ is called the \emph{square energy} of $G$. In 2016, Elphick, Farber, Goldberg and Wocjan conjectured that for every connected graph $G$ of order $n$,  $s(G)\geq n-1.$ No linear bound for $s(G)$ in terms of $n$ is known. 
Let $H_1, \ldots, H_k$ be disjoint vertex-induced subgraphs of $G$. In this note, we prove that 
\[s^+(G)\geq\sum_{i=1}^{k} s^+(H_i) \quad \text{ and } \quad s^-(G)\geq\sum_{i=1}^{k} s^-(H_i),\] which implies that $s(G)\geq \frac{3n}{4}$ for every connected graph $G$ of order $n\ge 4$.
\end{abstract}

\section{Introduction}

We use standard graph theory notation throughout the paper. All graphs are simple, i.e. with no loops or multiple edges. Let $G=(V, E)$ be a graph of order $n$ and size $m$. The \textit{adjacency matrix} of $G$ is an $n\times n$ matrix $A(G) = [a_{ij}]$, where $a_{ij} = 1$ if the vertices $v_i$ and $v_j$ are adjacent and $a_{ij} = 0$, otherwise. The \emph{eigenvalues} of $G$ are the eigenvalues of $A(G)$. Since $A(G)$ is a real symmetric matrix, all eigenvalues of $A(G)$ are real and can be listed as  
\[\lambda_1(G) \geq  \cdots \geq \lambda_n(G).\] 
Let $\mathcal{B}=\{x_1, \ldots , x_n\}$ be an orthonormal basis for $\mathbb{R}^n$ containing the eigenvectors of $A(G)$, where $x_i$ is the eigenvector corresponding to the eigenvalue $\lambda_i$ for $i=1, \ldots , n$. By the spectral decomposition (see \cite[Theorem 4.1.5]{Horn_Johnson_2013}), we have $A(G) = \sum_{i=1}^n \lambda_i x_ix_i^T$. Define
$$ A_{+}=\sum_{\lambda_i>0}\lambda_ix_ix_i^T,\,\, A_{-}=-\sum_{\lambda_i<0} \lambda_i  x_ix_i^T.$$
Both $A_{+}$ and $A_{-}$ are positive semidefinite matrices, and the following equalities hold: 
$$A_{+}A_{-}=A_{-}A_{+}=0, \,\,\,\, \, \, A(G)=A_{+}-A_{-}.$$
 
The \emph{energy} of a graph $G$, $\mathcal{E}(G)$, is defined to be the sum of absolute values of all eigenvalues of $G$, i.e. \[\mathcal{E}(G)=\sum_{i=1}^n|\lambda_i(G)|.\]
Define \[s^+(G)=\sum_{\lambda_i > 0}\lambda_i^2(G) \quad\text{ and }\quad s^-(G)=\sum_{\lambda_i < 0}\lambda_i^2(G).\] 
The parameters $s^+(G)$ and $s^-(G)$ are called the \emph{positive square energy} and the \emph{negative square energy} of $G$, respectively. Define $s(G)=\min\{s^+(G), s^-(G)\}$ and call it the \emph{square energy} of $G$. Clearly, $s^+(G)=\tr((A_{+})^2)$ and $s^-(G)=\tr((A_{-})^2)$.

Based on the fact that $s(G)=|E(G)|$ for every bipartite graph, Elphick, Farber, Goldberg and Wocjan \cite{Elphick_FGW_2016} proposed the following conjecture. 
\begin{conjecture}[\cite{Elphick_FGW_2016}] \label{conj:s_plus_main}
For every connected graph $G$ of order $n$, 
    \[ s(G)\ge n-1.\]
\end{conjecture}

The above conjecture has been verified for several graph classes, which include regular graphs, but the general case is wide open (see \cite{Elphick_Linz_2024, Aida_2023, Liu_Ning_OpenProblems_2023} for partial results). The best-known general lower bound for $s(G)$ is $\sqrt{n}$ as observed by Elphick and Linz \cite{Elphick_Linz_2024}, and is a consequence of a result on the chromatic number of graphs by Ando and Lin \cite{Ando_Lin_2015}. In this paper, our main result is a linear lower bound for the square energy of connected graphs. In particular, we show that for any connected graph $G$ of order $n\ge 4$, $s(G)\ge \frac{3n}{4}$. 

We believe the above bound can be improved to $\frac{4n}{5}$ using more intricate partitioning of the graph $G$ and applying Theorem \ref{thm:square_energy_partition}. We avoid doing this and content ourselves with the slightly weaker $\frac{3n}{4}$ bound because we believe more ideas are needed to resolve Conjecture \ref{conj:s_plus_main}.

\section{Main Results}

An important result concerning the energy of graphs is the following.
\begin{theorem}[\cite{Akbari_2009}]\label{thm:energy_partition_bound}
Let $H_1, \ldots, H_k$ be vertex-disjoint induced subgraphs of a graph $G$. Then 
\[ \mathcal{E}(G)\ge \sum_{i=1}^k \mathcal{E}(H_i).\]
\end{theorem}
We prove a similar result for the square energy of graphs.
\begin{theorem}\label{thm:square_energy_partition}
Let $H_1, \ldots, H_k$ be disjoint vertex induced subgraphs of a graph $G$. Then 
$$ s^+(G)\geq \sum_{i=1}^{k} s^+(H_i)\ \text{ and }\  s^-(G)\geq \sum_{i=1}^{k} s^-(H_i),$$
and equality holds in both simultaneously if and only if $G$ is the disjoint union of $H_1, \ldots, H_k$. 
\end{theorem}
\begin{proof}
It is sufficient to prove the assertion when $G$ is partitioned into two disjoint vertex induced subgraphs $H_1$ and $H_2$. So let 
$A(G) = \begin{bmatrix}
A_{11} & A_{12}\\ A_{21} & A_{22} \end{bmatrix}$, where $A_{11}$ and $A_{22}$ are the adjacency matrices of the induced subgraphs $H_1$ and $H_2$, respectively. We show that $$ s^+(G)\geq s^+(H_1) + s^+(H_2).$$
If we apply this inequality to $-A$, we get the second inequality. 

Let $A_{+}=[B_{ij}]$ and $A_{-}=[C_{ij}]$, $1\leq i,j\leq 2$, partitioned conformally as $A(G)$. We have $A_{ii}=B_{ii}-{C_{ii}}$, for $i=1,2$. Since $A_{+}$ and $A_{-}$ are positive semidefinite matrices,
both $B_{ii}$ and $C_{ii}$ are also positive semidefinite for $i=1,2$ (see \cite[Theorem 7.7.7]{Horn_Johnson_2013}). Now, we have
\[s^+(G)= \tr((A_{+})^2)  =\tr(B_{11}^2) + \tr(B_{22}^2) + 2\tr(B_{12}B_{12}^T).\]
Since $B_{12}B_{12}^T$ is a positive semidefinite matrix, we have 
\[ \tr(B_{12}B_{12}^T)\ge 0.\]
 Since $B_{ii}=A_{ii}+C_{ii}$ and $C_{ii}$ is a positive semidefinite matrix, $\lambda_r(B_{ii}) \geq \lambda_r(A_{ii})$ for $1\le r \le p_i$, where $p_i$ is the number of positive eigenvalues of $A_{ii}$ for $i=1,2$. This implies
 \[ s^+(G) \geq \tr(B_{11}^2) + \tr(B_{22}^2) \ge s^+(H_1) + s^+(H_2).\]

Note that if equality holds simultaneously then $\tr(B_{12}B_{12}^T)=0=\tr(C_{12}C_{12}^T)$ which implies $B_{12}=0=C_{12}$ and so $B_{21}=B_{12}^T=0=C_{12}^T=C_{21}$. Hence, $H_1$ and $H_2$ are disjoint. Conversely, if $G$ is the disjoint union of $H_1$ and $H_2$, then equality clearly holds. The proof is complete. 
\end{proof}

We recall the following well-known fact (cf. \cite[Theorem 4.3.17]{Horn_Johnson_2013}).

\begin{theorem}[Interlacing Theorem]\label{thm:Interlacing}
Let $A$ be a real symmetric matrix of order $n$. Let $B$ be a principal submatrix of $A$ of order $n-1$. Then, for $1\le i \le n-1$,
\[\lambda_i(A) \ge \lambda_i(B) \ge \lambda_{i+1}(A).\]
\end{theorem}

We now give a linear lower bound for the square energy of connected graphs using Theorem \ref{thm:square_energy_partition}.

\begin{theorem}
For any connected graph $G$ of order $n\ge 4$,
\[ s(G)\ge \frac{3n}{4}.\] 
\end{theorem}

\begin{proof}
For $n\le 10$, one can use computer to verify the stronger claim that $s(G)\ge n-1$. So, assume $n\ge 11$. We proceed by induction on $n$. 

The assertion is true if $G$ is a bipartite graph or a cycle (see \cite{Aida_2023}). So, assume $G$ is not bipartite and has maximum degree $\Delta \ge 3$. Let $T$ be a spanning tree of $G$ rooted at a vertex $v$, where $\deg_T(v)=\Delta$. If $\Delta=3$, we can find an edge $e$ in $T$ such that $T-e$ has two components $T_1$ and $T_2$, both of order at least 4 since $n\ge 11$. Using Theorem \ref{thm:square_energy_partition} and induction hypothesis, 
\[s(G)\ge s(G[V(T_1)])+s(G[V(T_2)])\ge \frac{3|V(T_1)|}{4} + \frac{3|V(T_2)|}{4}=\frac{3n}{4}.\] 

Now let $\Delta\ge 4$. Suppose $T-v$ has a component $C$ of order at least 4. Since $C$ and $T-V(C)$ are connected and have order at least 4, we are done by induction hypothesis. So, we may assume that every component of $T-v$ has at most 3 vertices. Hence, $T$ is a tree, as shown in Figure \ref{fig:tree_T}.

\begin{figure}
    \centering
\begin{tikzpicture}[scale=0.7, rotate=132]
\draw[thick]
 (0,0)-- (-1.9167713671725364,0.57095317320033)
 (0,0)-- (-0.8614591724298767,1.8049620755673599)
 (0,0)-- (0.8305846991792917,1.8193760077260681)
 (0,0)-- (1.914551282235562,0.5783540331753259)
 (0.8305846991792917,1.8193760077260681)-- (0.82,3.74)
 (0.8305846991792917,1.8193760077260681)-- (2.42,3.12)
 (1.914551282235562,0.5783540331753259)-- (3.84,0.58)
 (1.914551282235562,0.5783540331753259)-- (3.22,2.16)
 (0,0)-- (-1.816515682034847,-0.8368218310497613)
 (0,0)-- (-0.5896002011791943,-1.9111178934773891)
 (-1.816515682034847,-0.8368218310497613)-- (-3.98, -3.28)
 (-0.5896002011791943,-1.9111178934773891)-- (-2.58, -4.6)
 (0,0)-- (0.6507278289190346,-1.8911777527959448)
 (0,0)-- (1.8118628238009327,-0.8468489285156534)
 (0.6507278289190346,-1.8911777527959448)--(1.644359015525997,-3.1873225595702372)
 (1.8118628238009327,-0.8468489285156534)-- (2.9223360321996767,-2.0704817290471);

\draw[fill=black]
 (0,0) circle (2pt)
 (0,0.4) node {$v$}
 (-1.9167713671725364,0.57095317320033) circle (2pt)
 (-0.8614591724298767,1.8049620755673599) circle (2pt)
 (0.8305846991792917,1.8193760077260681) circle (2pt)
 (1.914551282235562,0.5783540331753259) circle (2pt)
 (0.82,3.74) circle (2pt)
 (2.42,3.12) circle (2pt)
 (3.84,0.58) circle (2pt)
 (3.22,2.16) circle (2pt)
 (-1.816515682034847,-0.8368218310497613) circle (2pt)
 (-0.5896002011791943,-1.9111178934773891) circle (2pt)
 (-3.98, -3.28) circle (2pt)
 (-2.58, -4.6) circle (2pt)
 (-2.86, -2.02) circle (2pt)
 (-1.6, -3.28) circle (2pt)
 (0.6507278289190346,-1.8911777527959448) circle (2pt)
 (1.8118628238009327,-0.8468489285156534) circle (2pt)
 (1.644359015525997,-3.1873225595702372) circle (2pt)
 (2.9223360321996767,-2.0704817290471) circle (2pt)
 (-1.2, 0.8) circle (1pt)
 (-0.95, 1.05) circle (1pt)
 (0.98,1.18) circle (1pt)
 (1.32,0.85) circle (1pt)
 (2,-1.82) circle (1pt)
 (1.58,-2.16) circle (1pt)
 (-1.9,-1.72) circle (1pt)
 (-1.48,-2.1) circle (1pt);
\end{tikzpicture}
    \caption{Spanning tree $T$}
    \label{fig:tree_T}
\end{figure}
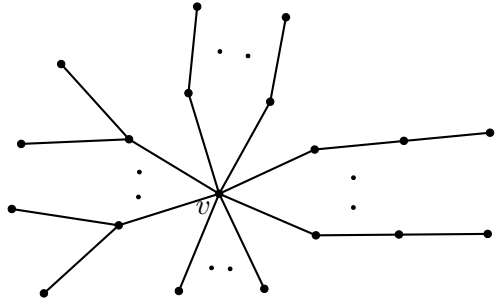

First, suppose that $G-v$ has a subgraph $H\cong P_4$. Let $e$ be an edge in $H$. Suppose $T-v+e$ has a component $C$ of order at least 4. Since $|V(C)|\le 6$, $G-V(C)$ is connected and has order at least 5. Thus, by induction hypothesis and Theorem \ref{thm:square_energy_partition}, $s(G)\ge \frac{3n}{4}$. So we may assume that every component of $T-v+e$ has order at most 3. This implies the component (say $C$) of the graph $T-v+E(H)$ (i.e., the graph $T-v$ with extra edges $E(H)$), which contains $H$ has order at most 6. Again, applying induction hypothesis on $C$ and $G-V(C)$ gives $s(G)\ge \frac{3n}{4}$. 

Now, consider the case that $G-v$ has no $P_4$ as a subgraph. Then $G-v$ is the disjoint union of some $K_3$, $P_3$, $K_2$ and $K_1$. Let the number of  $K_3$, $P_3$, $K_2$ and $K_1$ in $G-v$ be $\ell_1$, $\ell_2$, $\ell_3$ and $\ell_4$, respectively. Then $n=3\ell_1+ 3\ell_2 + 2\ell_3 + \ell_4 + 1$. Moreover, the spectrum of $G-v$ is 
\[\{ 2^{(\ell_1)}, \sqrt{2}^{(\ell_2)}, 1^{(\ell_3)}, 0^{(\ell_2+\ell_4)}, (-1)^{(2\ell_1+ \ell_3)}, (-\sqrt{2})^{(\ell_2)}\}.\]

Note that $G$ is not bipartite, so it contains an odd cycle. Since $G-v$ does not have $P_4$ as a subgraph, all odd cycles in $G$ are triangles. 

Now, if $n=11$, then we can find an edge $uw$ (where $u,w\neq v$) in $G$ such that $G-u-w$ is connected. Then, using the stronger claim for $n\le 10$, we have 
\[s(G)\ge s(K_2) + s(G-u-w)\ge n-2 \ge \frac{3n}{4}.\] 

Now assume $n\ge 12$. Note that $\min\{\lambda_1(G), |\lambda_n(G)|\}\ge \sqrt{\ell_1+\ell_2+\ell_3+\ell_4}$ since $G$ has an induced star of order $\ell_1+\ell_2+\ell_3+\ell_4+1$. Using the Interlacing Theorem,
\begin{align*}
    s^-(G) &\ge  \lambda_n^2(G) + s^-(G-v)- \lambda_{n-1}^2(G-v) \\
    & \ge (\ell_1+\ell_2+\ell_3+\ell_4) + (2\ell_1 + \ell_3 + 2\ell_2) - 2 \\
    & \ge n-3 \ge \frac{3n}{4}.
\end{align*}
This proves the assertion for $s^-(G)$. 

Now, if $\ell_1\ge 1$ then $\lambda_1(G-v)=2$. Using the Interlacing Theorem,
\begin{align*}
    s^+(G) &\ge  \lambda_1^2(G) + s^+(G-v)- \lambda_1^2(G-v) \\
    & \ge (\ell_1+\ell_2+\ell_3+\ell_4) + (4\ell_1 + 2\ell_2 + \ell_3) - 4 \\
    & = n-3 + (2\ell_1 - 2)\ge n-3 \ge \frac{3n}{4}.
\end{align*}
On the other hand, if $\ell_1=0$ then $\lambda_1(G-v)\le \sqrt{2}$. Again,
\begin{align*}
    s^+(G) &\ge  \lambda_1^2(G) + s^+(G-v)- \lambda_1^2(G-v) \\
    & \ge (\ell_2+\ell_3+\ell_4) + (2\ell_2 + \ell_3) - 2 \\
    & = n-3  \ge \frac{3n}{4}.
\end{align*}
This proves the assertion for $s^+(G)$ and the proof is complete.
\end{proof}  

\bibliographystyle{plain}
\bibliography{references}

\begin{thebibliography}{1}

\bibitem{Aida_2023}
Aida Abiad, Leonardo de~Lima, Dheer~Noal Desai, Krystal Guo, Leslie Hogben, and
  Jos\'e Madrid.
\newblock Positive and negative square energies of graphs.
\newblock {\em Electron. J. Linear Algebra}, 39:307--326, 2023.

\bibitem{Akbari_2009}
Saieed Akbari, Ebrahim Ghorbani, and Mohammad~Reza Oboudi.
\newblock Edge addition, singular values, and energy of graphs and matrices.
\newblock {\em Linear Algebra Appl.}, 430(8-9):2192--2199, 2009.

\bibitem{Ando_Lin_2015}
Tsuyoshi Ando and Minghua Lin.
\newblock Proof of a conjectured lower bound on the chromatic number of a
  graph.
\newblock {\em Linear Algebra Appl.}, 485:480--484, 2015.

\bibitem{Elphick_FGW_2016}
Clive Elphick, Miriam Farber, Felix Goldberg, and Pawel Wocjan.
\newblock Conjectured bounds for the sum of squares of positive eigenvalues of
  a graph.
\newblock {\em Discrete Math.}, 339(9):2215--2223, 2016.

\bibitem{Elphick_Linz_2024}
Clive Elphick and William Linz.
\newblock Symmetry and asymmetry between positive and negative square energies
  of graphs.
\newblock {\em Electron. J. Linear Algebra}, 40:418--432, 2024.

\bibitem{Horn_Johnson_2013}
Roger~A. Horn and Charles~R. Johnson.
\newblock {\em Matrix analysis}.
\newblock Cambridge University Press, Cambridge, second edition, 2013.

\bibitem{Liu_Ning_OpenProblems_2023}
Lele Liu and Bo~Ning.
\newblock Unsolved problems in spectral graph theory.
\newblock {\em Oper. Res. Trans.}, 27(4):33--60, 2023.

\end{thebibliography}

\vspace{0.6cm}
\noindent Saieed Akbari, Email: {\tt sakbarif@sfu.ca}\\
\textsc{Dept. of Mathematical Sciences, Sharif University of Technology, Tehran, Iran}\\[2pt]

\noindent Hitesh Kumar, Email: {\tt hitesh\_kumar@sfu.ca}\\
\textsc{Dept. of Mathematics, Simon Fraser University, Burnaby, BC \ V5A 1S6, Canada}\\[2pt]

\noindent Bojan Mohar, Email: {\tt mohar@sfu.ca}\\
Supported in part by the NSERC Discovery Grant R832714 (Canada), and in part by the ERC Synergy grant KARST (European Union, ERC, KARST, project number 101071836).
On leave from FMF, Department of Mathematics, University of Ljubljana.\\
\textsc{Dept. of Mathematics, Simon Fraser University, Burnaby, BC \ V5A 1S6, Canada}\\[2pt]

\noindent Shivaramakrishna Pragada, Email: {\tt shivaramakrishna\_pragada@sfu.ca}\\
\textsc{Dept. of Mathematics, Simon Fraser University, Burnaby, BC \ V5A 1S6, Canada}

\end{document}